\documentclass[11pt, british]{amsart}

\usepackage{babel,eucal,url,amssymb,
enumerate,amscd,
}

\newtheorem{thm}{Theorem}[section]

\newtheorem{pro}[thm]{Proposition}
\newtheorem{co}[thm]{Corollary}
\newtheorem{defn}[thm]{Definition}

\newcommand{\Gtwo}{\ifmmode{{\rm G}_2}\else{${\rm G}_2$}\fi}

\newcommand{\g}{\mathfrak{g}}

\newcommand{\F}{\mathcal{F}}
\newcommand{\R}{\mathbb{R}}

\newcommand{\ee}{\end{equation}}
\newcommand{\be}[1]{\begin{equation}\label{#1}}

\begin{document}

\title[]
 {Radical screen transversal half lightlike submanifolds of almost contact B-metric manifolds}


\author[G. Nakova]{Galia Nakova}

\address{
St. Cyril and St. Methodius University of Veliko Tarnovo, Bulgaria}
\email{gnakova@gmail.com}




\begin{abstract}
We introduce a class of half lightlike submanifolds of almost contact B-metric manifolds and prove that such submanifolds are semi-Riemannian  with respect to the associated B-metric.  
Object of investigations are also minimal of the considered submanifolds and a non-trivial example for them is given. 
\end{abstract}

\maketitle

\section{Introduction}\label{sec-1}
The general theory of lightlike submanifolds of semi-Riemannian manifolds has been developed by Duggal, Bejancu, Sahin in \cite{D-B, D-S}. Half lightlike submanifolds of indefinite almost contact metric manifolds have been studied by D. H. Jin in \cite{J, J2}, where different types of such submanifolds according to the behaviour of the almost contact structure were examined. Half lightlike submanifolds of almost contact B-metric manifolds have not been studied yet, as far as we know.
\par
On almost contact B-metric manifolds there exist two B-metrics $\overline g$ and its associated metric 
$\overline {\widetilde g}$. Therefore we can consider two induced metrics $g$ and $\widetilde g$ on their submanifolds by $\overline g$ and $\overline {\widetilde g}$, respectively. In this paper we define
and characterize geometrically a class of half lightlike submanifolds of almost contact B-metric manifolds, called radical screen transversal half lightlike submanifolds. A distinguishing feature of a subclass of the introduced submanifolds  is that with respect to the induced metric of the associated metric they are semi-Riemannian of codimension two. This result is presented in Theorem \ref{Theorem 1.1}:
\begin{thm}\label{Theorem 1.1}
Let $(\overline M,\overline \varphi,\overline \xi,\overline \eta,\overline g,\overline {\widetilde g})$ be a $(2n+1)$-dimensional almost contact B-metric manifold and $(M,g,S(TM),{\rm Rad} (TM))$ be an ascreen radical screen transversal half lightlike submanifold of 
$(\overline M,\overline g)$. Then 
\begin{enumerate}
\item[(i)] $(M ,\widetilde g)$ is a semi-Riemannian submanifold of 
$(\overline M,\overline {\widetilde g})$ of codimension two and the vector fields $N_1=\overline \xi -L$, $N_2=2\overline \xi-2\mu N-L$ form an orthonormal basis with respect to $\overline {\widetilde g}$ of the normal bundle 
$TM^{\widetilde \bot }$ of $(M ,\widetilde g)$ such that
\begin{equation}\label{1.1}
\overline {\widetilde g}(N_1,N_1)=-\overline {\widetilde g}(N_2,N_2)=1, \quad \overline {\widetilde g}(N_1,N_2)=0 .
\end{equation}
\item[(ii)] The tangent bundle $TM$ of $(M,\widetilde g)$ is an orthogonal direct sum
with respect to $\widetilde g$ of the distributions $S(TM)$ and ${\rm Rad} (TM)$. Moreover, both $S(TM)$ and ${\rm Rad} (TM)$ are non-degenerate with respect to $\widetilde g$, ${\rm Rad} (TM)$ is spacelike and the signature of $\widetilde g$ on $S(TM)$ is $(n-1,n-1)$.
\end{enumerate}
\end{thm}
We also initiate the simultaneous investigation of both submanifolds from Theorem \ref{Theorem 1.1}.
We start with studying of minimal of the considered submanifolds when the ambient manifold is an 
$\F_0$-manifold and we prove the following theorem:
\begin{thm}\label{Theorem 1.2}
Let $(\overline M,\overline \varphi,\overline \xi,\overline \eta,\overline g,\overline {\widetilde g})$ be a $(2n+1)$-dimensional $\F_0$-manifold and $(M,g)$ be an ascreen radical screen transversal half lightlike submanifold of 
$(\overline M,\overline g)$. Then the following assertions are equivalent:
\begin{enumerate}
[align=parleft]
\item[(i)] $(M,\widetilde g)$ is a minimal submanifold of $(\overline M,\overline {\widetilde g})$.
\item[(ii)] $(M,g)$ is a minimal submanifold of $(\overline M,\overline g)$.
\end{enumerate}
\end{thm}
We give an example that confirms Theorem \ref{Theorem 1.1} and Theorem \ref{Theorem 1.2} and note that the constructed submanifolds are proper minimal, i.e. they are not totally geodesic.
\section{Preliminaries}\label{sec-2}
A $(2n+1)$-dimensional smooth manifold  $(\overline M,\overline \varphi,\overline \xi,\overline \eta,\overline g)$ is called {\it an almost contact B-metric manifold}  \cite{GaMGri} if it is endowed with an almost contact structure $(\overline \varphi,\overline \xi,\overline \eta)$ consisting of an endomorphism $\overline \varphi $ of the tangent bundle, a vector field $\overline \xi $ and 1-form 
$\overline \eta $, satisfying the relations:
\begin{equation*}
\overline \varphi^2X=-X+\overline \eta(X)\overline \xi, \qquad \quad \overline \eta(\overline \xi)=1,
\end{equation*}
where $X\in T\overline M$. Also, $\overline M $ is equipped with a semi-Riemannian metric $\overline g$, called {\it a B-metric} \cite{GaMGri}, determined by
$\overline g(\overline \varphi X,\overline \varphi Y)=-\overline g(X,Y)+\overline \eta(X)\overline \eta(Y), \,  X,Y\in T\overline M$.
Immediate consequences of the above conditions are:
\begin{equation*}
\overline \eta \circ \overline \varphi =0, \quad \overline \varphi \overline \xi =0, \quad {\rm rank}(\overline \varphi)=2n, \quad \overline \eta (X)=
\overline g(X,\overline \xi ), \quad \overline g(\overline \xi,\overline \xi )=1.
\end{equation*}
The tensor field $\overline {\widetilde g}$ of type $(0,2)$ given by 
$\overline {\widetilde g}(X,Y)=\overline g(X,\overline \varphi Y)+\overline \eta (X)\overline \eta (Y)$
is a B-metric, called {\it an associated metric} to $\overline g$. Both metrics $\overline g$ and 
$\overline {\widetilde g}$ are necessarily of signature $(n+1,n)$. 
Throughout this paper, for the orthonormal basis the signature of the metric $g$ will be of the form $(+\ldots + -\ldots -)$.\\
Let $\overline \nabla$ be the Levi-Civita connection of $\overline g$. A classification of the almost contact B-metric manifolds with respect to the tensor $F(X,Y,Z)=\overline g((\overline \nabla_X\overline \varphi)Y,Z)$ is given in \cite{GaMGri} and eleven basic classes  $\F_i$ $(i=1,2,\dots,11)$ are obtained. If $(\overline M,\overline \varphi,\overline \xi,\overline \eta,\overline g)$ belongs to $\F_i$ then it is called an {$\F_i$-{\it manifold}.
The special class $\F_0$ is determined by the condition $F(X,Y,Z)=0$ and in this class we have
$\overline \nabla \overline \varphi =\overline \nabla \overline \xi =\overline \nabla \overline \eta =\overline \nabla \overline g=\overline \nabla \overline {\widetilde g}=0$.
Let $\overline {\widetilde \nabla }$ be the Levi-Civita connection of $\overline {\widetilde g}$. 
In \cite{GaMGri} it is shown that 
the Levi-Civita connections $\overline \nabla $ and $\overline {\widetilde \nabla }$ of an $\F_0$-manifold coincide.
\par
In the remainder of this section we briefly recall the main notions about half lightlike submanifolds of semi-Riemannian manifolds for which we refer to \cite{D-B, D-S, J}.\\
If $(M,g)$ is a lightlike submanifold of $(\overline M,\overline g)$ then 
both the tangent space $T_xM$ and the normal space $T_xM^\bot  $, $x\in M$, are degenerate orthogonal subspaces of $T_x\overline M$ but they are not complementary. The intersection of $T_xM$ and $T_xM^\bot $ is denoted by  ${\rm Rad} (T_xM)$ and it is called {\it a radical subspace} of $T_x\overline M$. For a lightlike submanifold $M$ {\it the radical distribution} 
${\rm Rad} (TM): x\in M \longrightarrow {\rm Rad} (T_xM)$
is of a constant rank. 
\begin{defn}\label{Definition 2.1}
A lightlike submanifold $(M,g)$ of codimension 2 of $(\overline M,\overline g)$ is called
{\it a half lightlike submanifold} if ${\rm rank}\, {\rm Rad} (TM)=1$. 
\end{defn}
\noindent
For a half lightlike submanifold $(M,g)$ there exist two complementary non-degenerate distributions 
$S(TM)$ and $S(TM^\bot)$ of ${\rm Rad} (TM)$ in the tangent bundle $TM$ and the normal bundle $TM^\bot $, respectively. Thus, the following decompositions are valid:
\begin{equation}\label{2.4}
TM={\rm Rad} (TM)\bot S(TM) , \quad TM^\bot ={\rm Rad} (TM)\bot S(TM^\bot ),
\end{equation}
where the symbol $\bot $ denotes the orthogonal direct sum.
The distributions $S(TM)$ and $S(TM^\bot)$ are called {\it a screen distribution} and {\it a screen transversal bundle} of $M$, respectively. Since ${\rm Rad} (TM)$ is a 1-dimensional subbundle of $TM^\bot $, the screen transversal bundle $S(TM^\bot)$ is also a 1-dimensional subbundle of $TM^\bot $. We choose $L$ as a unit vector field of $S(TM^\bot)$ and put 
$\overline g(L,L)=\epsilon $, where $\epsilon =\pm 1$. It is well known \cite{D-B, D-S} that 
for any $\xi \in \Gamma ({\rm Rad} (TM))$ there exists a unique locally defined vector field $N$ satisfying
$
\overline g(N,\xi )=1, \quad \overline g(N,N)=\overline g(N,L)=\overline g(N,X)=0,  \, \forall X\in \Gamma (S(TM))
$.
The 1-dimensional vector bundle ${\rm ltr}(TM)$ locally represented by the lightlike vector field $N$ is called {\it the lightlike transversal bundle} of $M$ with respect to the screen distribution $S(TM)$. The {\it transversal vector  bundle} ${\rm tr}(TM)$ of $M$ with respect to $S(TM)$ is the complementary (but never orthogonal) vector bundle to $TM$ in $T\overline M$ such that
${\rm tr}(TM)=S(TM^\bot)\bot {\rm ltr}(TM)$.
Thus, for $T\overline M$ we have
\begin{equation}\label{2.6}
T\overline M=TM\oplus {\rm tr}(TM)=\{{\rm Rad} (TM)\oplus {\rm ltr}(TM)\}\bot S(TM)\bot S(TM^\bot),
\end{equation}
where $\oplus $ denotes a non-orthogonal direct sum.
Denote by $P$ the projection of $TM$ on $S(TM)$, from the first decomposition in  \eqref{2.4} for any
$X \in \Gamma (TM)$ we obtain $X=PX+\eta (X)\xi $, where $\eta $ is a differential 1-form on $M$ given by $\eta (X)=\overline g(X,N)$.
\par
The local Gauss-Weingarten formulas of $(M,g)$ and $S(TM)$ are given by
\begin{equation}\label{2.8}
\begin{array}{l}
\overline \nabla _XY=\nabla _XY+B(X,Y)N+D(X,Y)L ,
\end{array}
\end{equation}
\begin{equation}\label{2.9}
\begin{array}{l}
\overline \nabla _XN=-A_NX+\tau (X)N+\rho (X)L ,
\end{array}
\end{equation}
\begin{equation}\label{2.10}
\begin{array}{l}
\overline \nabla _XL=-A_LX+\phi (X)N ;
\end{array}
\end{equation}
and $\nabla _XPY=\nabla ^*_XPY+C(X,PY)\xi $, \, $\nabla _X\xi =-A^*_\xi X-\tau (X)\xi $, 
$X, Y \in \Gamma (TM)$. The induced connections $\nabla $ and 
$\nabla ^*$ on $TM$ and $S(TM)$, respectively, are linear connections; $A_N$, $A_L$ and
$A^*_\xi $ are the shape operators on $TM$ and $S(TM)$ and $\tau $, $\rho $, 
$\phi $ are 1-forms on $TM$. The local second fundamental forms $B$ and $D$ are called 
{\it the lightlike second fundamental form} and {\it the screen second fundamental form} of 
$M$, respectively, and $C$ -- {\it the local second fundamental form} of $S(TM)$. The local second fundamental forms $B$, $C$ and $D$ are related to their shape operators as follows:
\begin{equation*}
B(X,Y)=g(A^*_\xi X,Y), \, \overline g(A^*_\xi X,N)=0; \, \, C(X,PY)=g(A_N X,PY), \, \overline g(A_N X,N)=0;
\end{equation*}
\begin{equation*}
\epsilon D(X,PY)=g(A_L X,PY), \quad  \overline g(A_L X,N)=\epsilon \rho (X),  
\end{equation*}
\begin{equation*}
\epsilon D(X,Y)=g(A_L X,PY)-\phi (X)\eta (Y), \qquad \forall X, Y\in \Gamma (TM).
\end{equation*}
Since $\overline \nabla $ is torsion-free, $\nabla $ is also torsion-free. Therefore $B$ and $D$ are symmetric $F(M)$-bilinear forms on $\Gamma (TM)$. Also we have
\begin{equation*}
B(X,\xi )=0, \quad D(X,\xi )=-\phi (X), \quad \forall X\in \Gamma (TM).
\end{equation*}
In general, the induced connection $\nabla $ is not metric and satisfies
\begin{equation*}
(\nabla _Xg)(Y,Z)=B(X,Y)\eta (Z)+B(X,Z)\eta (Y).
\end{equation*}
The linear connection $\nabla ^*$ is not torsion-free but it is a metric connection on $S(TM )$.
The shape operators $A^*_\xi $ and $A_N$ are $\Gamma (S(TM))$-valued, $A^*_\xi $ is self-adjoint with respect to $g$ and $A^*_\xi \xi =0$.
\section{Radical screen transversal half lightlike submanifolds of almost contact B-metric manifolds}\label{sec-3}
Let $(M,g,S(TM))$ be a half lightlike submanifold of $(\overline M,\overline \varphi,\overline \xi,\overline \eta,\overline g)$, where $\overline M$ is an indefinite almost contact metric manifold or an almost contact B-metric manifold. In both cases, taking into account \eqref{2.6}, we have the following decomposition for $\overline \xi$:
\begin{equation}\label{3.1}
\overline \xi =\xi _0+a\xi +bN+cL ,
\end{equation}
where $\xi _0$ is a smooth vector field on $S(TM)$ and $a, b, c$ are smooth functions on $\overline M$. According to the decomposition \eqref{3.1} of $\overline \xi$, tangential and ascreen half lightlike submanifolds of indefinite cosymplectic and indefinite Sasakian manifolds have been studied in \cite{J, J2}. Analogously as in \cite{J, J2}, we define such submanifolds when the ambient manifold is an almost contact B-metric manifold. A half lightlike submanifold $M$ of an almost contact B-metric manifold $\overline M$ is said to be:
{\it tangential} if $\overline \xi $ is tangent to $M$;
{\it ascreen} if $\overline \xi $  belongs to ${\rm Rad} (TM)\oplus {\rm ltr} (TM)$.
The considered tangential and ascreen half lightlike submanifolds in \cite{J, J2} are equipped with screen distributions $S(TM)$ such that 
$\overline \varphi ({\rm Rad} (TM))$,   
$\overline \varphi ({\rm ltr} (TM))$ and $\overline \varphi (S(TM^\bot ))$ belong to $S(TM)$.
Such $S(TM)$ is called {\it a generic screen distribution}.
In this section we define a  half lightlike submanifold of an almost contact B-metric manifold with a non-generic screen distribution and prove that this submanifold is non-tangential.
\begin{defn}\label{Definition 3.3}
We say that a half lightlike submanifold $M$ of an almost  contact B-metric manifold $(\overline M,\overline \varphi,\overline \xi,\overline \eta,\overline g)$ is {\it a Radical Screen Transversal  Half Lightlike (RSTHL) submanifold} of $\overline M$ if $\overline \varphi ({\rm Rad} (TM))=S(TM^\bot )$.
\end{defn}
\begin{pro}\label{Proposition 3.1}
Let $(M,g)$ be an RSTHL submanifold of an almost contact B-metric manifold 
$(\overline M,\overline g)$ and $S(TM)={\rm span} \{L\}$, where $L$ is a unit vector field. Then
\begin{enumerate}
\item[(i)] $L$ is orthogonal to $\overline \xi $ and it is spacelike (i.e. $\overline g(L,L)=1$).
\item[(ii)] $\overline \xi \in TM\oplus {\rm ltr} (TM)$ but $\overline \xi $ does not belong neither to $TM$, nor to ${\rm ltr} (TM)$.
\end{enumerate}
\end{pro}
\begin{proof}
(i) Let $\xi \in \Gamma ({\rm Rad} (TM))$. According to Definition~\ref{Definition 3.3} we have
\begin{equation}\label{3.2}
\overline \varphi \xi =\mu L,
\end{equation}
where $\mu $ is a non-zero smooth function on $M$. From \eqref{3.2} it follows
\begin{equation}\label{3.3}
\overline \eta (L)=0 .
\end{equation}
Now, from $\overline g(\overline \varphi \xi ,\overline \varphi \xi )=\mu ^2\overline g(L,L)$ and
$\overline g(\overline \varphi \xi ,\overline \varphi \xi )=\overline \eta ^2(\xi )$ we derive
\begin{equation}\label{3.4}
\overline g(L,L)=1,
\end{equation}
\begin{equation}\label{3.5}
\overline \eta (\xi )=\pm \mu ,
\end{equation}
which completes the proof.\\
(ii) For the functions $a$, $b$, $c$ in \eqref{3.1}, taking into account \eqref{3.3} and \eqref{3.5},  we get
$a=\overline \eta (N), \quad b=\overline \eta (\xi )=\pm \mu , \quad c=\overline \eta (L)=0$.
Then \eqref{3.1} becomes
\begin{equation}\label{3.6}
\overline \xi =\xi _0+\overline \eta (N)\xi \pm \mu N,
\end{equation}
which means that $\overline \xi \in TM\oplus {\rm ltr} (TM)$.
Since $\mu \neq 0$, from \eqref{3.6} it is clear that $\overline \xi $ is not tangent to $M$. Let us assume that in \eqref{3.6} the tangential part of $\overline \xi $ is zero. Then  
$\overline g(\overline \xi ,\overline \xi )=0$, which is a contradiction. Thus, $\overline \xi $ does not belong to ${\rm ltr} (TM)$.
\end{proof}
As an immediate consequence of the assertion (ii) of Proposition~\ref{Proposition 3.1} 
we state
\begin{co}
There exist ascreen RSTHL submanifolds of almost contact B-metric manifolds and $\overline \xi $ does not belong neither to ${\rm Rad} (TM)$, nor to ${\rm ltr} (TM)$.
\end{co}

\section{Ascreen RSTHL submanifolds of almost contact B-metric manifolds. Proof of Theorem \ref{Theorem 1.1}}\label{sec-4}  
We note that for an ascreen RSTHL submanifold of an almost contact B-metric manifold
the equalities \eqref{3.2}, \eqref{3.3}, \eqref{3.4} and \eqref{3.5} hold. Without loss of generality we assume that $\overline \eta (\xi )=\mu $. Now, replacing $\xi _0$ in \eqref{3.6} with $0$, we obtain                 
\begin{equation}\label{4.1}
\overline \xi =\overline \eta (N)\xi +\mu N.
\end{equation}
By using \eqref{4.1} and the equality $\overline g(\overline \xi ,\overline \xi )=1$, we find
\begin{equation}\label{4.2}
\overline \eta (N)=1/2\mu  .
\end{equation}
Substituting \eqref{4.2} in \eqref{4.1} we have
\begin{equation}\label{4.3}
\overline \xi =(1/2\mu )\xi +\mu N.
\end{equation}
Applying $\overline \varphi $ to the both sides of \eqref{4.3}, using that $\overline \varphi \overline \xi =0$ and \eqref{3.2}, we get
\begin{equation}\label{4.4}
\overline \varphi N=-(1/2\mu )L .
\end{equation}
Now, applying $\overline \varphi $ to the both sides of \eqref{3.2} we obtain
\begin{equation}\label{4.5}
\overline \varphi L=-(1/2\mu )\xi +\mu N .
\end{equation}
Using \eqref{3.2}, \eqref{4.3},\eqref{4.4} and \eqref{4.5} we state
\begin{pro}\label{Proposition 4.1}
Let $(M,g,S(TM))$ be an ascreen RSTHL submanifold of an almost contact B-metric manifold 
$(\overline M,\overline \varphi,\overline \xi,\overline \eta,\overline g)$. Then the following assertions are fulfilled:
\begin{enumerate}
[align=parleft]
\item[(i)] $\overline \varphi ({\rm ltr} (TM))=S(TM^\bot )$, \, $\overline \varphi (S(TM))=S(TM)$.
\item[(ii)]  $\overline \varphi (S(TM^\bot ))\in {\rm Rad} (TM)\oplus {\rm ltr} (TM)$ but 
$\overline \varphi (S(TM^\bot ))$ does not coincide neither with ${\rm Rad} (TM)$, nor with ${\rm ltr} (TM)$.
\end{enumerate}
\end{pro}
{\bf Proof of Theorem \ref{Theorem 1.1}.} (i): First, by using $\overline \varphi \overline \xi =0$,  \eqref{4.4} and \eqref{4.5}, we get
\begin{equation}\label{4.7}
\overline \varphi N_1=(1/2\mu )\xi -\mu N, \quad \overline \varphi N_2=(1/2\mu )\xi -\mu N+L.
\end{equation}
With the help of \eqref{3.3}, \eqref{4.2} and $\overline \eta (\overline \xi )=1$ we find
\begin{equation}\label{4.8}
\overline \eta (N_1)=\overline \eta (N_2)=1 .
\end{equation}
By direct calculations, applying \eqref{4.7} and \eqref{4.8}, we verify that $N_1$ and $N_2$ satisfy the equalities \eqref{1.1}. Take $X\in \Gamma (TM)$ and exploiting \eqref{4.1} we have
\begin{equation}\label{4.9}
\overline \eta (X)=\mu \eta (X).
\end{equation}
By employing \eqref{4.7}, \eqref{4.8} and \eqref{4.9} we obtain $\overline {\widetilde g}(X,N_1)=\overline {\widetilde g}(X,N_2)=0$ for any $X\in \Gamma (TM)$, which means that the vector fields $N_1$ and $N_2$ are normal to $(M,\widetilde g)$. Since the normal bundle $TM^{\widetilde \bot }$ of $(M ,\widetilde g)$ is of dimension 2, it follows that the orthonormal pair $\{N_1,N_2\}$ with respect to $\overline {\widetilde g}$ form a basis of $TM^{\widetilde \bot }$. According to 
\eqref{1.1}, the signature of $\overline {\widetilde g}$ on $TM^{\widetilde \bot }$ is $(1,1)$. Therefore $\overline {\widetilde g}$ is non-degenerate on  $TM^{\widetilde \bot }$.
Now, let us assume that the induced metric $\widetilde g$ on $M$ by $\overline {\widetilde g}$ is degenerate. Then there exists a vector field $U\in \Gamma (TM)$ such that $U\neq 0$ and 
$\widetilde g(U,X)=0$ for any $X\in \Gamma (TM)$. Hence, $U$ belongs to the normal bundle 
$TM^{\widetilde \bot }$ of $(M,\widetilde g)$. Then, taking into account that $\overline {\widetilde g}(U,N_1)=\overline {\widetilde g}(U,N_2)=0$ and $\widetilde g(U,U)=0$, we conclude that  
$\overline {\widetilde g}$ is degenerate on $TM^{\widetilde \bot }$, which is a contradiction. So we established that  $(M,\widetilde g)$ is a semi-Riemannian submanifold of $\overline M$ of codimension two.\\
(ii): Direct calculations show that $\widetilde g(X,\xi )=0$ and $\widetilde g(\xi ,\xi )=\mu ^2$ for any $X\in \Gamma S(TM)$ and $\xi \in {\rm Rad} (TM)$ Thus, it follows that
$TM=S(TM)\widetilde \bot {\rm Rad} (TM)$ ($\widetilde \bot $ denotes the orthogonal direct sum with respect to $\widetilde g$) and ${\rm Rad} (TM)$ is spacelike. The latter assertions and $TM$ is non-degenerate imply that $S(TM)$ is also non-degenerate. Taking into account that the signature of 
$\widetilde g$ on $(M,\widetilde g)$ is $(n,n-1)$ and ${\rm Rad} (TM)$ is spacelike, we complete the proof.
\par
From Theorem \ref{Theorem 1.1} and $\overline \varphi S(TM)=S(TM)$ it is clear that  there exists an orthonormal basis 
$\{e_1,\ldots ,e_{n-1},\overline \varphi e_1,\ldots ,\overline \varphi e_{n-1}\}$ of $S(TM)$ with respect to $\widetilde g$ and 
\begin{equation}\label{4.10}
\left\{(1/\mu )\xi ,e_1,\ldots ,e_{n-1},\overline \varphi e_1,\ldots ,\overline \varphi e_{n-1}\right\}
\end{equation}
is an orthonormal basis of $TM$ with respect to $\widetilde g$.
\par
At the end of this section we provide the following results:
\begin{pro}\label{Proposition 5.1}
For an ascreen RSTHL submanifold $(M,g)$ of $\overline M\in \F_0$ we have: 
\begin{enumerate}
\item[(i)] $A_NX=-(1/2\mu ^2)A^*_\xi X$, \quad $A_LX=(1/\mu )\overline \varphi (A^*_\xi X)$, \\
$D(X,Y)=(1/\mu )B(X,\overline \varphi (PY)), \quad C(X,PY)=-(1/2\mu ^2)B(X,Y)$, \\
$\tau (X)=-X(\mu )/\mu $, \quad $\phi (X)=\rho (X)=0$, for any  $X, Y \in \Gamma (TM)$.
\item[(ii)] The shape operators $A^*_\xi$, $A_N$ and $A_L$ commute with $\overline \varphi $ on $S(TM)$.
\item[(iii)] $B(\overline \varphi X,\overline \varphi Y)=-B(X,Y)$, \, $\forall X,Y\in S(TM)$. 
\end{enumerate}
\end{pro}
From now on in this paper,  $(\overline M,\overline \varphi,\overline \xi,\overline \eta,\overline g,
\overline {\widetilde g})$ is an $\F_0$-manifold and $(M,g)$, $(M,\widetilde g)$ are the submanifolds of $\overline M$ from Theorem \ref{Theorem 1.1}.
\section{Relations between the induced geometric objects on the submanifolds $(M,g)$ and $(M,\widetilde g)$ of $\F_0$-manifolds. \\
Proof of Theorem \ref{Theorem 1.2}}\label{sec-6}
Applying Proposition \ref{Proposition 5.1}, the  formulas \eqref{2.8}, \eqref{2.9} and \eqref{2.10} for the submanifold $(M,g)$ of $(\overline M,\overline g)\in \F_0$ become
\begin{align}\label{5.9}
\begin{aligned}
\overline \nabla _XY=\nabla _XY+B(X,Y)N+(1/\mu )B(X,\overline \varphi (PY))L , \\
\overline \nabla _XN=(1/2\mu ^2)A^*_\xi X-(X(\mu )/\mu )N, \quad \overline \nabla _XL=-(1/\mu )\overline \varphi (A^*_\xi X) .
\end{aligned}
\end{align}
On the other hand, for the Gauss-Weingarten formulas of a non-degenerate submanifold $(M,\widetilde g)$ of codimension two of a semi-Riemannian manifold $(\overline M,\overline {\widetilde g})$ we have 
\begin{align}\label{5.12}
\begin{aligned}
\overline {\widetilde \nabla }_XY=\widetilde \nabla _XY+h_1(X,Y)N_1+h_2(X,Y)N_2 , \\
\overline {\widetilde \nabla }_XN_1=-\widetilde A_{N_1}X+\alpha (X)N_2, \quad \overline {\widetilde \nabla }_XN_2=-\widetilde A_{N_2}X+\alpha (X)N_1 , 
\end{aligned}
\end{align}
$\forall X,Y\in \Gamma (TM)$ ,where: $\widetilde \nabla $ is the Levi-Civita connection of $\widetilde g$; $N_1$ and 
$N_2$ are normal vector fields to $(M,\widetilde g)$ satisfying \eqref{1.1};
$\widetilde h(X,Y)=h_1(X,Y)N_1+h_2(X,Y)N_2$ is the second fundamental form of $(M,\widetilde g)$; 
$\widetilde A_{N_1}$ and $\widetilde A_{N_2}$ are  the shape operators with respect to $N_1$ and 
$N_2$, respectively; $\alpha $ is a 1-form on $(M,\widetilde g)$. The bilinear  forms $h_1$, $h_2$,
given by $h_1(X,Y)=\widetilde g(\widetilde A_{N_1}X,Y)$, \, $h_2(X,Y)=-\widetilde g(\widetilde A_{N_2}X,Y)$, are symmetric. Therefore $\widetilde A_{N_1}$ and $\widetilde A_{N_2}$ are self-adjoint with respect to $\widetilde g$.
\par 
By using that $\overline \nabla =\overline {\widetilde \nabla }$ for an 
$\F_0$-manifold and \eqref{5.9}, \eqref{5.12} we obtain the following 
\begin{pro}\label{Proposition 6.1}
The induced geometric objects on the submanifolds $(M,g)$ and $(M,\widetilde g)$ of an $\F_0$-manifold $(\overline M,\overline \varphi,\overline \xi,\overline \eta,\overline g,\overline {\widetilde g})$ are related as follows:
\begin{equation*}
\widetilde \nabla _XY=\nabla _XY+(1/\mu ^2)\left((1/2)B(X,Y)+B(X,\overline \varphi (PY))\right)\xi ,
\end{equation*}
\begin{equation}\label{5.17}
h_1(X,Y)=(1/\mu )B(X,Y), \quad h_2(X,Y)=-(1/\mu )(B(X,Y)+B(X,\overline \varphi (PY))),
\end{equation}
\begin{equation*}
\widetilde A_{N_1}X=-(1/\mu )\overline \varphi (A^*_\xi X), \quad
\widetilde A_{N_2}X=(1/\mu )(A^*_\xi X-\overline \varphi (A^*_\xi X)), \quad \alpha (X)=0.
\end{equation*}
\end{pro}
It is known that a semi-Riemannian submanifold $(M,\widetilde g)$ is  minimal if 
${\rm trace}_{\widetilde g}\widetilde h=0$.
We recall that a lightlike submanifold 
$(M,g)$ is minimal \cite{B-D}  if $h^S=0$ on ${\rm Rad} (TM)$ and ${\rm trace}_{g_{|S(TM)}}h=0$.\\
{\bf Proof of Theorem \ref{Theorem 1.2}.} 
From \eqref{5.9} and \eqref{5.12}, \eqref{5.17} for the second fundamental forms $h$ and $\widetilde h$ of $(M,g)$ and $(M,\widetilde g)$, respectively, we have
\begin{equation}\label{5.23}
h(X,Y)=B(X,Y)N+(1/\mu )B(X,\overline \varphi (PY))L ,
\end{equation}
\begin{equation}\label{5.24}
\widetilde h(X,Y)=(-3/2\mu ^2)B(X,Y)\xi +B(X,Y)N+(1/\mu )B(X,\overline \varphi (PY))L 
\end{equation}
for any $X,Y\in \Gamma (TM)$. By using \eqref{5.23}, \eqref{5.24} and (iii) of Proposition
\ref{Proposition 5.1} we obtain
\begin{equation}\label{5.25}
h(\overline \varphi X,\overline \varphi Y)=-h(X,Y), \quad \widetilde h(\overline \varphi X,\overline \varphi Y)=-\widetilde h(X,Y)
\end{equation}
for any $X,Y\in \Gamma S(TM)$. Take the orthonormal basis \eqref{4.10} of $TM$, we find  
\[
{\rm trace}_{\widetilde g}\widetilde h=\widetilde h\left((1/\mu )\xi ,(1/\mu )\xi\right)+\sum_{i=1}^{n-1}(\widetilde h(e_i,e_i)-
\widetilde h(\overline \varphi e_i,\overline \varphi e_i)) .
\]
The second equality in \eqref{5.25} and $\widetilde h(\xi ,\xi)=0$ imply ${\rm trace}_{\widetilde g}\widetilde h=2\sum_{i=1}^{n-1}\widetilde h(e_i,e_i)$. From \eqref{5.24} and \eqref{2.6} it follows that ${\rm trace}_{\widetilde g}\widetilde h=0$ if and only if
\begin{equation}\label{5.26} 
\sum_{i=1}^{n-1}B(e_i,e_i)=0, \qquad \sum_{i=1}^{n-1}B(e_i,\overline \varphi e_i)=0.
\end{equation}
Thus, we established that (i) is equivalent to the conditions  \eqref{5.26}. 
Now we will show that (ii) is also equivalent to \eqref{5.26}. Since $h^S(\xi ,\xi )=D(\xi ,\xi )=0$, (ii) is equivalent to the condition ${\rm trace}_{g_{|S(TM)}}h=0$. The following system of vector fields 
\[
\left\{(e_1+\overline \varphi e_1)/\sqrt{2},\ldots ,(e_{n-1}+\overline \varphi e_{n-1})/\sqrt{2},(\overline \varphi e_1-e_1)/\sqrt{2}\ldots ,
(\overline \varphi e_{n-1}-e_{n-1})/\sqrt{2}\right\}
\]
is an orthonormal basis of $S(TM)$ with respect to $g$, where $\{e_i,\overline \varphi e_i\}$ 
$(i=1,\ldots ,n-1)$ are the vector fields from \eqref{4.10}. With the help of the above basis and \eqref{5.23} we get
\begin{equation*}
\begin{array}{lll}
{\rm trace}_{g_{|S(TM)}}h
=2\displaystyle\sum_{i=1}^{n-1}h(e_i,\overline \varphi e_i)
=2\displaystyle\sum_{i=1}^{n-1}\left(B(e_i,\overline \varphi e_i)N-(1/\mu )B(e_i,e_i)L\right).
\end{array}
\end{equation*}
The above equality and \eqref{2.6} imply that ${\rm trace}_{g_{|S(TM)}}h=0$ if and only if the conditions \eqref{5.26} hold. Thus, we complete the proof.
\section{A Lie subgroup as a minimal ascreen RSTHL submanifold and as a minimal semi-Riemannian submanifold of a 7-dimensional Lie group as an $\F_0$-manifold}\label{sec-7}
Let $\overline G$ be a 7-dimensional real connected Lie group and let $\overline {\g}$ be its Lie algebra. If 
$\{e_i\}$ $(i=1,\ldots ,7)$ is a global basis of left invariant vector fields of $\overline G$, we define an almost contact structure $(\overline \varphi,\overline \xi,\overline \eta)$ and a left invariant B-metric 
$\overline g$ on $\overline G$ as follows:
\begin{align}\label{7.1}
\begin{aligned}\overline \varphi e_i=e_{i+3}, \, \overline \varphi e_{i+3}=-e_i\,  (i=1,2,3), \,
\overline \varphi e_7=0; \, \overline \xi =e_7; \, \overline \eta (e_7)=1;\\
\overline \eta (e_i)=0\,  (i=1,2,3,4,5,6), \, \overline g(e_i,e_j)=0, i\neq j, \, i,j\in \{1,2,3,4,5,6,7\},\\
\overline g(e_i,e_i)=\overline g(e_7,e_7)=-\overline g(e_{i+3},e_{i+3})=1 \, \, (i=1,2,3).
\end{aligned}
\end{align}
Thus, $(\overline G,\overline \varphi,\overline \xi,\overline \eta ,\overline g)$ is a 7-dimensional almost contact B-metric manifold with  an orthonormal basis  $\{\overline \xi ,e_1,e_2,e_3,\overline \varphi e_1,\overline \varphi e_2,\overline \varphi e_3\}$ of  $\overline {\g}$. Let the Lie algebra $\overline {\g}$ of $\overline G$ be determined by the following non-zero commutators:
\begin{align}\label{7.2}
\begin{aligned}
\left[e_1,e_2\right]=-[\overline \varphi e_1,\overline \varphi e_2]=\lambda _1e_2+\lambda _2e_3+\lambda _3\overline \varphi e_2+\lambda _4\overline \varphi e_3, \\
[e_1,e_3]=-[\overline \varphi e_1,\overline \varphi e_3]=\lambda _5e_2-\lambda _1e_3+\lambda _6\overline \varphi e_2-\lambda _3\overline \varphi e_3, \\
[e_2,\overline \varphi e_1]=[\overline \varphi e_2,e_1]=\lambda _3e_2+\lambda _4e_3-\lambda _1\overline \varphi e_2-\lambda _2\overline \varphi e_3, \\
[e_3,\overline \varphi e_1]=[\overline \varphi e_3,e_1]=\lambda _6e_2-\lambda _3e_3-\lambda _5\overline \varphi e_2+\lambda _1\overline \varphi e_3,
\end{aligned}
\end{align}
where $\lambda _i\in {\R}$ $(i=1,2,3,4,5,6)$. From \eqref{7.2} it is clear that  any vector fields $X, Y$ of $\overline G$ satisfy the condition $[\overline \varphi X,\overline \varphi Y]=-[X,Y]$, i.e. the almost contact structure $(\overline \varphi ,\overline \xi ,\overline \eta )$ on $\overline G$ is {\it non-Abelian}
\cite{MM}. It is known \cite{MM} that a Lie group $(\overline G,\overline \varphi ,\overline \xi  ,\overline \eta ,\overline g)$ which is an almost contact B-metric manifold with a non-Abelian almost contact structure is an $\F_0$-manifold} iff $[X,Y]=-\overline \varphi [\overline \varphi X,Y]$. Directly we check that the commutators in \eqref{7.2} satisfy the latter condition. Hence, $(\overline G,\overline \varphi,\overline \xi,\overline \eta ,\overline g)$ belongs to the class $\F_0$.
\par 
Let us consider the subspace ${\g}$ of $\overline {\g}$ spanned by $\{e_2,e_3,\overline \varphi e_2,\overline \varphi e_3,\xi =-\mu \overline \varphi e_1+\mu \overline \xi \}$, $\mu \in{\R}, \mu \neq 0$. By using \eqref{7.2} we check that ${\g}$ is a Lie subalgebra of $\overline {\g}$. Hence, the corresponding  to ${\g}$ Lie subgroup $G$ of $\overline G$ is a 5-dimensional submanifold of
$\overline G$. The induced metric $g$ on $G$ by $\overline g$ is degenerate and ${\rm Rad} ({\g})=
{\rm span} \{\xi \}$, which means that $(G,g)$ is a half lightlike submanifold of $(\overline G,\overline g)$. We take 
the screen distribution $S({\g})$, the lightlike transversal bundle ${\rm ltr} ({\g})$ and the screen transversal bundle $S({\g}^\bot )$ of $(G,g)$ as follows: $S({\g})={\rm span}\{e_2,e_3,\overline \varphi e_2,\overline \varphi e_3\}$, ${\rm ltr} ({\g})={\rm span}\{N=(1/2\mu )(\overline \varphi e_1+\overline \xi )\}$, $S({\g}^\bot )={\rm span}\{L=e_1\}$. Since $\overline \varphi \xi =\mu L$ and $\overline \xi =(1/2\mu )\xi +\mu N$, it follows that $(G,g)$ is an ascreen RSTHL submanifold.
\par
If $\widetilde g$ is the induced metric on $G$ by the associated metric $\overline {\widetilde g}$, the determinant of the matrix of $\widetilde g$ with respect to the basis $\{e_2,e_3,\overline \varphi e_2,\overline \varphi e_3,\xi \}$ of ${\g}$ is $\mu ^2\neq 0$. Hence, $(G,\widetilde g)$ is a non-degenerate submanifold of $(\overline G,\overline {\widetilde g})$. It is easy to see that the normal bundle of $(G,\widetilde g)$ is spanned by $N_1=\overline \xi -e_1$, $N_2=\overline \xi -e_1-\overline \varphi e_1$ and they satisfy \eqref{1.1}.
\par
First, by standard calculations we find the following components of the Levi-Civita connection $\overline \nabla$ of  
$\overline g$:
$\overline \nabla _{e_2}e_2=-\overline \nabla _{\overline \varphi e_2}
\overline \varphi e_2=-\overline \nabla _{e_3}e_3=\overline \nabla _{\overline \varphi e_3}
\overline \varphi e_3=\lambda _1e_1+\lambda _3\overline \varphi e_1$ and 
$\overline \nabla _\xi \xi  =0$. Then we obtain $h(e_2,e_2)=-h(\overline \varphi e_2,\overline \varphi e_2)=-h(e_3,e_3)=h(\overline \varphi e_3,\overline \varphi e_3)=\mu \lambda _3N+\lambda _1L$ and $h(\xi ,\xi )=0$. Since $\{e_2,e_3,\overline \varphi e_2,\overline \varphi e_3\}$ is an orthonormal basis of $S({\g})$ and the signature of $g_{|S(TM)}$ is $(+,+,-,-)$, we have ${\rm trace}_{g_{|S(TM)}}h=h(e_2,e_2)+h(e_3,e_3)-h(\overline \varphi e_2,\overline \varphi e_2)-h(\overline \varphi e_3,\overline \varphi e_3)=0$. Therefore $(G,g)$ is minimal.
\par
The basis $\{(1/\mu )\xi ,(e_2-\overline \varphi e_2)/\sqrt{2},(e_3-\overline \varphi e_3)/\sqrt{2},
(e_2+\overline \varphi e_2)/\sqrt{2},(e_3+\overline \varphi e_3)/\sqrt{2}\}$ of ${\g}$ is orthonormal with respect to $\widetilde g$ and the signature of $\widetilde g$ is $(+,+,+,-,-)$. By using this basis we get ${\rm trace}_{\widetilde g}\widetilde h=\widetilde h((1/\mu )\xi ,(1/\mu )\xi )-2(\widetilde h(e_2,\overline \varphi e_2)+\widetilde h(e_3,\overline \varphi e_3))$. From the components $\overline {\widetilde \nabla }_{e_2}
\overline \varphi e_2=-\overline {\widetilde \nabla }_{e_3}
\overline \varphi e_3=-\lambda _3e_1+\lambda _1\overline \varphi e_1$ and $\overline {\widetilde \nabla }_\xi \xi =0$ of the Levi-Civita connection 
$\overline {\widetilde \nabla }$ of  $\overline {\widetilde g}$ we derive $\widetilde h(e_2,\overline \varphi e_2)=-\widetilde h(e_3,\overline \varphi e_3)=\lambda _1N_1+(\lambda _3-\lambda _1)N_2$ and $\widetilde h((1/\mu )\xi ,(1/\mu )\xi )=0$. Thus, we obtain ${\rm trace}_{\widetilde g}\widetilde h=0$, i.e. $(G,\widetilde g)$ is minimal.


\begin{thebibliography}{20}

\bibitem{B-D} Bejan, C. L.,  Duggal, K. L.: \emph{Global lightlike manifolds and harmonicity}, Kodai Math. J., {\bf 28}, 131--145 (2005)

\bibitem{D-B} Duggal, K. L., Bejancu, A.: \emph{Lightlike Submanifolds of Semi-Riemannian Manifolds and Applications}. Kluwer Academic, 364 (1996)

\bibitem{D-S} Duggal, K. L., Sahin, B.: \emph{Differential Geometry of Lightlike 
Submanifolds}. (2010)
\bibitem{J} Jin, D. H.: \emph{Special half lightlike submanifolds of an indefinite cosymplectic manifold}, Journal of Function Spaces and Applications, doi:10.1155/2012/636242, (2012)

\bibitem{J2} Jin, D. H.: \emph{Special half lightlike submanifolds of an indefinite Sasakian manifold}, Commun. Korean Math. Soc., {\bf 29}, 109--121 (2014)



\bibitem{GaMGri}
Ganchev, G., Mihova, V., Gribachev, K.: \emph{Almost contact manifolds with B-metric}, Math. Balkanica {\bf 7}, 262--276 (1993)

\bibitem{MM} Manev, M.: \emph{Natural connection with totally skew-symmetric torsion on almost contact manifolds with B-metric}, IJGMMP, Vol. 9, {\bf 5}, 1250044 (20 pages) (2012)

\end{thebibliography}
\end{document}